\documentclass[12pt]{article}
\usepackage{amsmath,amssymb,amsthm}
\usepackage{url}
\usepackage{leftidx}
\newtheorem{thm}{Theorem}
\newtheorem{lem}{Lemma}
\newtheorem{prop}{Proposition}
\newtheorem{cor}{Corollary}
\newtheorem*{thm*}{Theorem}
\newtheorem*{prop*}{Proposition}
\newtheorem*{cor*}{Corollary}
\newtheorem*{conj*}{Conjecture}
\theoremstyle{definition}

\theoremstyle{definition}

\newtheorem*{rmk*}{Remark}
\theoremstyle{definition}

\newcommand{\ZZ}{\mathbb Z}
\newcommand{\NN}{\mathbb N}
\newcommand{\RR}{\mathbb R}
\newcommand{\QQ}{\mathbb Q}
\newcommand{\CC}{\mathbb C}
\newcommand{\HH}{\mathbb H}

\def\tr{\mathop{\mathrm{tr}}\nolimits}

\begin{document}

\author{Hanno von Bodecker\footnote{Fakult{\"a}t f{\"u}r Mathematik, Universit{\"a}t Bielefeld, Germany}}

\title{An analytical formula for the $f$--invariant of circle transfers}
%\title{Circle transfers, twisted $\eta$--invariants, and the Adams--Novikov 2--line}

\date{}

\maketitle
\begin{abstract}
In this note, we explain how the $f$--invariant of a circle transfer can be computed on the framed manifold itself in terms of the spectral asymmetry of twisted Dirac operators on the base. Some explicit examples and a treatment of the quaternionic case are provided as well.
\end{abstract}

\section{Introduction}
% the geometric circle transfer
Let $\lambda$ be a hermitian line bundle over a closed framed manifold $M$. Geometrically, the circle transfer takes $M$  to the sphere bundle $S(\lambda)$ with the induced framing; passing to bordism classes, this gives rise to a morphism of $\Omega^{fr}_{*}$--modules \cite{Loffler:1974qf}:
\begin{equation}\label{circletransfer}
S_{\CC}\colon\Omega^{fr}_{*}\left(\CC P^{\infty}\right)\rightarrow\Omega^{fr}_{*+1}.
\end{equation}
Knapp has shown that the circle transfer \eqref{circletransfer} is equivalent to considering the bistable complex $J$--homomorphism (up to sign), hence circle transfers may be used to represent the bordism classes of unitarily reframed boundaries; unfortunately, the bistable complex $J$--homomorphism is  not surjective (not even away from the prime two) \cite{Knapp:1979bh}.

% the question to be answered
As a first step towards identifying the framed bordism class of a given circle transfer, L\"offler and Smith have shown how to determine Adams' complex $e$--invariant of such a transfer: The associated disk bundle naturally carries the structure of $(U,fr)$--manifold (see \cite{Conner:1966jw} for a definition), and  its relative Todd genus can then be computed on $M$ itself using the Thom isomorphism \cite{Loffler:1974qf}.  Now suppose that $M$ is of dimension $2l-1$; then $e_{\CC}[S(\lambda)]=0$ and the next step would be to determine Laures' $f$--invariant \cite{Laures:1999sh,Laures:2000bs},
\begin{equation}\label{f}
f\colon\pi^{S}_{2l}S^{0}\rightarrow\underline{\underline{D}}^{\Gamma}_{l+1}\otimes{\mathbb{Q/Z}}.
\end{equation}
To this end, note that $\Omega^{U}_{2l-1}\left(\CC P^{\infty}\right)=0$; thus, the circle bundle $S(\lambda)$ can be extended over some $(U,fr)$--manifold $B$. Promoting the associated disk bundle to a $(U,fr)^{2}$--manifold, $f[S(\lambda)]$ could be computed by evaluating certain relative cohomology classes. In practice however, such a $(U,fr)$--manifold $B$ is hard to find, so an intrinsic  description, i.e.\ one computable in terms of the original datum, is desirable.

% APS
In \cite{Atiyah:1975ai}, Atiyah, Patodi, and Singer have shown how  the (real) $e$--invariant of a closed parallelized manifold (of dimension $4k-1$) can be computed on the manifold itself: First of all, a parallelism determines a Riemannian metric, so besides the tautological flat connection annihilating the frame, there is the Levi-Civita connection; in particular, this can be used to define a secondary characteristic form associated to the $\hat A$--genus, which is essentially the invariant considered by Chern and Simons \cite{Chern:1974qo}. Furthermore, the parallelism determines a spin structure, so there is a canonical Riemannian Dirac operator $\eth$; its spectral asymmetry is encoded in the quantity $(\eta(\eth)+\dim\ker(\eth))/2$ that occurs as a boundary correction in the APS--index theorem. Thus, modulo the integers, the relative $\hat A$--genus is determined by the combination of these two quantities.

% organization of this note
The purpose of this note is to establish a similar description of the $f$--invariant for the circle transfer (and  also for its quaternionic analog); it is organized as follows: The first step is to establish a suitable   expansion of the Hirzebruch elliptic genus. Then, we derive an intrinsic formula for the $f$--invariant of the circle transfer on a closed parallelized manifold $M$; it involves the $\eta$--invariants of twisted Dirac operators, where the twisting runs over all non-trivial powers of the line bundle $\lambda$, but we refer to Section \ref{complex case} for the precise statement. 
In Section \ref{quaternionic case}, we deal with the quaternionic case. To illustrate our results,  we treat several examples in Section \ref{some examples}. Finally, a brief discussion of the modifications required in the case of a merely stably parallelized manifold (followed by an explicit example) is included as an appendix.

%%%%%%%%%%%%%%%%%%%%%%%%%%%%%%%%%%%%%%%%%%%%%%%%%%%

\section{An expansion of the elliptic genus}\label{an expansion}
The first thing to set up is a power expansion of the Hirzebruch elliptic genus of level $N>1$; our approach is  along the lines of Zagier's work  \cite{Zagier:1988jy}:
\begin{prop}\label{bernoulli}
The power series associated to the Hirzebruch elliptic genus of level $N$ is given by
\begin{equation*}
Ell^{\Gamma_{1}(N)}(x)=1+\sum_{k\geq1}\widehat{G}_{k}^{(N)}(\tau)\frac{x^{k}}{(k-1)!}
\end{equation*}
with coefficients $q$--expanding as
\begin{equation*}
\widehat{G}_{k}^{(N)}(\tau)=c_{k}-\sum_{n\geq1}\left(\sum_{d|n}\left(\zeta^{-n/d}+(-1)^{k}\zeta^{n/d}\right)d^{k-1}\right)q^{n}
\end{equation*}
where $\zeta=\exp(2\pi i/N)$, $c_{1}=\frac{1}{2}+\frac{\zeta}{1-\zeta}$, and $c_{k}=B_{k}/k$ for $k>1$.
\end{prop}

\begin{proof} 
For $x\in\CC$ and fixed $\tau\in\mathfrak{h}$ consider the series
\begin{equation}\label{near zero}
\psi(x)=\frac{1}{2}\coth\left(\frac{x}{2}\right)+\sum_{n\geq1}\zeta^{n}\frac{q^{n}e^{-x}}{1-q^{n}e^{-x}}-\sum_{n\geq1}\zeta^{-n}\frac{q^{n}e^{x}}{1-q^{n}e^{x}}
\end{equation} 
which is absolutely convergent for $|q|<\textrm{min}\{|e^{x}|,|e^{-x}|\}$. Now, for any integer $l>0$, we have
\begin{equation}\label{positive shift}
\begin{split}
&\quad\ \psi^{}(x+2\pi il\tau)\\
&=\frac{1}{2}\left(\frac{q^{l}e^{-x}+1}{q^{l}e^{-x}-1}\right)+\sum_{n\geq1}\zeta^{n}\frac{q^{n-l}e^{-x}}{1-q^{n-l}e^{-x}}-\sum_{n\geq1}\zeta^{-n}\frac{q^{n+l}e^{x}}{1-q^{n+l}e^{x}}\\
&=-\frac{1}{2}-\frac{q^{l}e^{x}}{1-q^{l}e^{x}}+\sum_{n\geq l+1}\zeta^{n}\frac{q^{n-l}e^{-x}}{1-q^{n-l}e^{-x}}\\
&\quad\ +\zeta^{l}\frac{e^{-x}}{1-e^{-x}}+\sum_{n=1}^{l-1}\zeta^{n}\frac{q^{n-l}e^{-x}}{1-q^{n-l}e^{-x}}-\sum_{n\geq1}\zeta^{-n}\frac{q^{n+l}e^{x}}{1-q^{n+l}e^{x}}\\
&=\zeta^{l}\frac{1}{e^{x}-1}-\frac{1}{2}-\sum_{n=1}^{l-1}\zeta^{n}+\sum_{n\geq l+1}\zeta^{n}\frac{q^{n-l}e^{-x}}{1-q^{n-l}e^{-x}}\\
&\quad\ -\zeta^{0}\frac{q^{l}e^{x}}{1-q^{l}e^{x}}-\sum_{n=1}^{l-1}\zeta^{-(-n)}\frac{q^{l-n}e^{x}}{1-q^{l-n}e^{x}}-\sum_{n\geq1}\zeta^{-n}\frac{q^{l+n}e^{x}}{1-q^{l+n}e^{x}}\\
\end{split}
\end{equation}
and
\begin{equation}\label{negative shift}
\begin{split}
&\quad\ \psi(x-2\pi il\tau)\\
&=\frac{1}{2}\left(\frac{1+q^{l}e^{-x}}{1-q^{l}e^{-x}}\right)+\sum_{n\geq1}\zeta^{n}\frac{q^{n+l}e^{-x}}{1-q^{n+l}e^{-x}}-\sum_{n\geq1}\zeta^{-n}\frac{q^{n-l}e^{x}}{1-q^{n-l}e^{x}}\\
&=\zeta^{-l}\frac{e^{x}}{e^{x}-1}+\frac{1}{2}+\sum_{n=1}^{l-1}\zeta^{-n}-\sum_{n-l\geq 1}\zeta^{-n}\frac{q^{n-l}e^{x}}{1-q^{n-l}e^{x}}\\
&\quad+\zeta^{0}\frac{q^{l}e^{-x}}{1-q^{l}e^{-x}}+\sum_{n=1}^{l-1}\zeta^{(-n)}\frac{q^{l-n}e^{-x}}{1-q^{l-n}e^{-x}}+\sum_{n\geq1}\zeta^{n}\frac{q^{l+n}e^{-x}}{1-q^{l+n}e^{-x}}\\
\end{split}
\end{equation}
Clearly, the expressions \eqref{near zero}, \eqref{positive shift}, and \eqref{negative shift} remain unchanged if we replace $x$ by $x+2\pi i k$ for  $k\in\ZZ$, so $\psi(x)$ extends to  a well-defined meromorphic function on $\CC$. Moreover, $\psi(x)$ has only simple poles, which lie on the lattice $2\pi i(\ZZ+\tau\ZZ)$, and  is easily seen to be doubly periodic w.r.t.\ the sublattice $2\pi i(\ZZ+N\tau\ZZ)$.
On the other hand, the $\Phi$--function transforms as follows  (see e.g.\ \cite{Hirzebruch:1992dw})
\begin{equation*}
\Phi\left(\tau,x+2\pi i(\lambda\tau+\mu)\right)=q^{\lambda^{2}/2}e^{-\lambda x}(-1)^{\lambda+\mu}\Phi(\tau,x)\quad \forall\lambda,\mu\in\ZZ,
\end{equation*}
implying that
\begin{equation*}
\frac{Ell^{\Gamma_1\left(N\right)}\left(x\right)}{x}=\frac{\Phi(\tau,x-2\pi i/N)}{\Phi(\tau,x)\Phi(\tau,-2\pi i/N)}
\end{equation*}
is also doubly periodic w.r.t.\ $2\pi i(\ZZ+N\tau\ZZ)$ and has the same poles and residues as $\psi(x)$, hence these two functions must agree up to a (possibly $\tau$--dependent) constant; considering the logarithmic derivative, one easily deduces
\begin{equation*}
Ell^{\Gamma_{1}(N)}(x)=1+\left(\frac{1}{2}+\frac{\zeta}{1-\zeta}+\sum_{m\geq1}\left(\sum_{d|m}2\sinh\left(\frac{2\pi id}{N}\right)\right)q^{m}\right)x+O(x^{2}),
\end{equation*}
so expanding
\begin{equation}\label{qexp near zero}
\begin{split}
\psi(x)&=\frac{1}{2}\coth\left(\frac{x}{2}\right)+\sum_{n\geq1}\left(\zeta^{n}\frac{q^{n}e^{-x}}{1-q^{n}e^{-x}}-\zeta^{-n}\frac{q^{n}e^{x}}{1-q^{n}e^{x}}\right)\\
&=\frac{1}{2}\coth\left(\frac{x}{2}\right)+\sum_{n\geq1}\left(\zeta^{n}\sum_{k\geq1}q^{kn}e^{-kx}-\zeta^{-n}\sum_{k\geq1}q^{kn}e^{kx}\right)\\
&=\frac{1}{2}\coth\left(\frac{x}{2}\right)+\sum_{m\geq1}\left(\sum_{d|m}\left(\zeta^{m/d}e^{-dx}-\zeta^{-m/d}e^{dx}\right)\right)q^{m}
\end{split}
\end{equation}
w.r.t.\ $x$ implies $\psi(x)+\textstyle{\frac{1}{2}+\frac{\zeta}{1-\zeta}}=Ell^{\Gamma_{1}(N)}(x)/x$ and yields the desired result.
\end{proof}
Working at a fixed level $N>1$, we are going to suppress the subscript $\Gamma_{1}(N)$ from the notation if confusion is unlikely; furthermore, we are going to use the abbreviations
\begin{equation*}
Ell_{0}=Ell|_{q=0},\quad \widetilde{Ell}=Ell-Ell_0.
\end{equation*}

% key observation
An interesting observation concerning Proposition \ref{bernoulli} is:
\begin{prop}\label{reduced genus}
For a complex line bundle $\lambda$, we have
\begin{equation}
\frac{\widetilde{Ell}^{\Gamma_{1}(N)}(\lambda)}{c_{1}(\lambda)}=-\sum_{n\geq1}\left(\sum_{d|n}\left(\zeta^{-n/d}ch(\lambda^{d})-\zeta^{n/d}ch(\lambda^{-d})\right)\right)q^{n}.
\end{equation}
\end{prop}
\begin{proof} This follows immediately from \eqref{qexp near zero}.
\end{proof}

%%%%%%%%%%%%%%%%%%%%%%%%%%%%%%%%%%%%%%%%%%%%%%%%%%%

\section{The complex transfer}\label{complex case}
% definition of f
Regarding the definition of the $f$--invariant, let us briefly recall some notation from \cite{Laures:1999sh,Laures:2000bs}: Considering the congruence subgroup $\Gamma=\Gamma_1(N)$ for a fixed level $N$, let $\mathbb{Z}^{\Gamma}=\mathbb{Z}[\zeta_N,1/N]$ and denote by $M^{\Gamma}_*$ the graded ring of modular forms w.r.t.~$\Gamma$ which expand integrally, i.e.\ which lie in $\mathbb{Z}^{\Gamma}[\![q]\!]$. The ring of  divided congruences $D^{\Gamma}$  consists of those rational combinations of modular forms which expand integrally; this ring can be filtered by setting
\begin{equation*}
D_k^{\Gamma}=\left\{\left.f={\textstyle{\sum_{i=0}^{k}}}f_i\ \right| f_i\in M_i^{\Gamma}\otimes\mathbb{Q},\ f\in\mathbb{Z}^{\Gamma}[\![q]\!]\right\}.
\end{equation*}
Furthermore, put
\begin{equation*}\underline{\underline{D}}^{\Gamma}_{k}=D^{\Gamma}_k+M_0^{\Gamma}\otimes\mathbb{Q}+M_k^{\Gamma}\otimes\mathbb{Q}.
\end{equation*}
For our purposes, a $(U,fr)^{2}$--manifold is just a stably almost complex $\langle2\rangle$--manifold $Z$  together with a stable decomposition $TZ\cong E_{1}\oplus E_{2}$ (in terms of complex vector bundles)  and specified trivializations of  the restriction of the $E_{i}$ to the respective face $\partial_{i}Z$ for $i=1,2$. Now let $Z$ be $2l+2$--dimensional; then we may consider the following evaluation of relative characteristic classes
\[F(Z)=\left\langle \widetilde{Ell}(E_{1})\left(Ell_{0}(E_{2})-1\right),[Z,\partial Z]\right\rangle;\]
clearly, this yields a rational combination of modular forms of weight $\leq l+1$, i.e.\ an element of $D_{l+1}^{\Gamma}\otimes\QQ$. It turns out that the class
\begin{equation}\label{f as a class}
\left[F(Z)\right]\in \underline{\underline{D}}^{\Gamma}_{l+1}\otimes\QQ/\ZZ
\end{equation}
is a well-defined invariant of the framed bordism class of the codimension-two corner of $Z$. In this note, we adopt the sign convention that the corner is to be considered as the boundary of the face $\partial_{2}Z$ (where, as usual,  boundaries inherit their orientation using the ``outward normal first'' convention) and that its $f$--invariant (at the level $N$) is then defined by \eqref{f as a class}; moreover, any fiber bundle (in particular $S(\lambda)$) will be oriented ``base first''.

% integrality
Next, we note the following two results:
\begin{lem}\label{integrality}
Let $B$ be an even-dimensional $(U,fr)$--manifold. Then, for any complex vector bundle $E$ over B, we have:
\begin{equation*}
\left\langle\left(Ell_{0}(TB)-Td(TB)\right)ch(E),[B,\partial B]\right\rangle\in\ZZ^{\Gamma}.
\end{equation*}
\end{lem}
\begin{proof} Clearly, we have $\left(Ell_{0}(TB)-Td(TB)\right)ch(E)=Td(TB)ch(E')$ for a relative class $E'\in K(B,\partial B)\otimes\ZZ^{\Gamma}$; thus, integrality follows from the Lemma on the bottom of page 119 of \cite{Stong:1968zp}.
\end{proof}
%\begin{rmk} Rewritten in terms of characteristic forms for compatible connections, this result can be deduced from \cite[Lemma 2.1]{Bodecker:2008pi}.
%\end{rmk}

% change of sign
\begin{lem}\label{change of sign}
Let $\lambda$ be a complex line bundle over a $(U,fr)$--manifold $B$ of dimension $2l$. Then we have:
\begin{equation*}
\begin{split}
&\quad\left\langle\widetilde{Ell}(TB)\frac{\left(Ell_{0}(\lambda)-1\right)}{c_{1}(\lambda)},[B,\partial B]\right\rangle\\
&\equiv-\left\langle\left(Td(TB)-1\right)\frac{\widetilde{Ell}(\lambda)}{c_{1}(\lambda)},[B,\partial B]\right\rangle\mod \underline{\underline{D}}_{l+1}^{\Gamma}.
\end{split}
\end{equation*}
\end{lem}
\begin{proof} Combining Lemma \ref{integrality} and Proposition \ref{reduced genus}, we may replace $Td(TB)$ by $Ell_{0}(TB)$ in the lower line. Moreover, the disk bundle $D(\lambda)$ becomes a $(U,fr)^{2}$--manifold in a natural way; thus, under the Thom isomorphism, the upper line can be identified with a representative of $f[S(\lambda|_{\partial B})]$. But for the $f$--invariant, interchanging the r\^oles of $\widetilde{Ell}$ and $Ell_{0}$ is known to result in a change of sign (modulo the indeterminacy), cf.\  \cite[Lemma 4.1.7]{Laures:2000bs}.
\end{proof}
% note that we do not have to resort to \cite[Lemma 4.1.7]{Laures:2000bs}:
\begin{rmk*}
More directly, note that we have the following equality of relative characteristic classes on a $(U,fr)^{2}$--manifold $Z$ of dimension $2l+2$:
\begin{equation*}
\begin{split}
&\widetilde{Ell}(E_1)(Ell_0(E_2)-1)+\widetilde{Ell}(E_2)(Ell_0(E_1)-1)+\widetilde{Ell}(E_1)\widetilde{Ell}(E_2)\\
&=\widetilde{Ell}(E_1\oplus E_{2})-\widetilde{Ell}(E_{1})-\widetilde{Ell}(E_{2}).\\
\end{split}
\end{equation*}
Thus, the evaluation of the upper line on the fundamental class $[Z,\partial Z]$ produces an element in $M_0^{\Gamma}\otimes\mathbb{Q}+M_{l+1}^{\Gamma}\otimes\mathbb{Q}$. Specializing to the disk bundle considered in the proof of the previous proposition, we have
$$\langle \widetilde{Ell}(E_1)\widetilde{Ell}(E_2),[Z,\partial Z]\rangle=\langle \widetilde{Ell}(TB)\widetilde{Ell}(\lambda)/c_{1}(\lambda),[B,\partial B]\rangle.$$
A priori, this expression lies in $D^{\Gamma}_{l+1}\otimes\QQ$. However, by combining  Proposition \ref{reduced genus} with an obvious modification of Lemma \ref{integrality}, we see that it expands integrally, thereby establishing the change-of-sign behavior.
\end{rmk*}

% main result, complex case
Now we can prove the result promised in the introduction:
\begin{thm}\label{main}
Let $M$ be a closed manifold of dimension $2l-1$, equipped with a parallelism $\pi$, and let $\lambda$ be a hermitian line bundle over $M$. Then, choosing a unitary connection $\nabla^{\lambda}$, the  $f$--invariant \textup{(}at the level $N$\textup{)} of the circle transfer $S(\lambda)$  may be deduced from the following formula:
\begin{multline}\label{levelN}
f[S(\lambda)]\equiv\sum_{n\geq1}\left(\sum_{d|n}\left(\zeta^{-n/d}\xi\left(\eth\otimes\lambda^{d},\pi\right)-\zeta^{n/d}\xi\left(\eth\otimes\lambda^{-d},\pi\right)\right)\right)q^{n}\\
\mod\underline{\underline{D}}_{l+1}^{\Gamma}+\ZZ^{\Gamma}[\![q]\!]+\RR\tilde G_{l+1}^{(N)}
\end{multline}
where $\tilde G_{k}^{(N)}=\widehat{G}_{k}^{(N)}-B_{k}/k$, and
$$\xi(\eth\otimes\lambda,\pi)={\textstyle\frac{1}{2}}(\eta(\eth\otimes\lambda)+\dim\ker(\eth\otimes\lambda))-\int_{M}cs\left(\hat{A};\nabla^{LC},\nabla^{\pi}\right)ch(\nabla^{\lambda})$$ consists of the Atiyah--Patodi--Singer spectral invariant of the $\lambda$--twisted Riemannian Dirac operator on $M$  and the Chern--Simons term relating the Levi--Civita connection to the parallelizing one. 
\end{thm}
\begin{proof} We have $\Omega^{U}_{2l-1}\left(\CC P^{\infty}\right)=0$ as an immediate consequence of the Atiyah--Hirzebruch spectral sequence, so the parallelized manifold $M$ can be realized as the boundary of some $(U,fr)$--manifold $B$ to which the line bundle extends (where, by abuse of notation, it  will given the same name $\lambda$). 

Assume for the moment that $B$ is an honest almost complex manifold; clearly, $TB$ may be equipped with a hermitian metric   and compatible connection $\nabla^{TB}$ such that they are of product type near the boundary and restrict to those determined by the parallelism on $TM$. Furthermore, considered as a complex vector bundle, $TB$ defines a complex line bundle $L$ by taking the highest exterior power, and this determinant line $L$ naturally inherits a hermitian metric and a compatible connection with curvature $F^{L}$. In particular, the relative characteristic class $(Td(TB)-1)$ can be represented by the differential form 
\[Td\left(\nabla^{TB}\right)-1=\hat A\left(\nabla^{TB}\right)\exp{\textstyle\left(\frac{iF^L}{4\pi}\right)}-1.\]
Obviously, the RHS is already determined by the connection on the underlying real bundle $TB$ and the connection on the determinant line $L$, so we will use this to {\em define} the LHS on a stably almost complex manifold in terms of metric connections that are chosen compatibly near the boundary (ignoring the choices notationally). Now we choose a unitary connection on $\lambda$ (which will then induce unitary connections on all tensor powers), and use Lemma \ref{change of sign} and Proposition \ref{reduced genus} to obtain:
\begin{equation*}
\begin{split}
&\quad\ f[S(\lambda)]\\
&\equiv\sum_{n\geq1}\left(\sum_{d|n}\left\langle (Td(TB)-1)\left(\zeta^{-n/d}ch(\lambda^{d})-\zeta^{n/d}ch(\lambda^{-d})\right),[B,\partial B]\right\rangle\right) q^{n}\\
&=\sum_{n\geq1}\left(\sum_{d|n}\int_{B} (Td\left(\nabla^{TB}\right)-1)\left(\zeta^{-n/d}ch(\nabla^{\lambda^{d}})-\zeta^{n/d}ch(\nabla^{\lambda^{-d}})\right)\right) q^{n}\\
&\equiv\sum_{n\geq1}\left(\sum_{d|n}\int_{B} Td\left(\nabla^{TB}\right)\left(\zeta^{-n/d}ch(\nabla^{\lambda^{d}})-\zeta^{n/d}ch(\nabla^{\lambda^{-d}})\right)\right) q^{n}\ {\rm mod}\ \RR\tilde G_{l+1}^{(N)}
\end{split}
\end{equation*}
Note that due to the indeterminacy introduced in the last line, the choice of  $\nabla^{\lambda}$ is immaterial. If, on the other hand, we want to change the connection on $TB$ to the Levi-Civita connection, we have to include a correction term: Using the transgression formula, we can construct a canonical form $cs(\hat A;\nabla^{TB,LC},\nabla^{TB})$ of (inhomogeneous) odd  degree satisfying
\[{\rm d}cs\left(\hat A;\nabla^{TB,LC},\nabla^{TB}\right)=\hat{A}\left(\nabla^{TB,LC}\right)-\hat{A}\left(\nabla^{TB}\right);\]
restriction to $M$ then yields the Chern--Simons form $cs(\hat A;\nabla^{LC},\nabla^{\pi})$ associated to $\hat A$. But since the curvature of the determinant line vanishes on $M$, it can also be interpreted as the Chern--Simons form associated to the Todd genus form. Thus, in order to obtain the desired correction term, it suffices to multiply with the Chern character form and integrate over $M$. Finally, we apply the Atiyah--Patodi--Singer index theorem, and the proof is complete.
%
% I.e. we use
%
%\[\int_{B}\hat A\left(\nabla^{TB,LC}\right)\exp{\textstyle\left(\frac{iF^L}{4\pi}\right)}ch(\nabla^{\lambda})\equiv\frac{\eta(\eth\otimes\lambda)+\dim\ker(\eth\otimes\lambda)}{2}\mod\ZZ\]
\end{proof}

Although it is  the conceptually correct approach to consider all twistings, it suffices to consider positive tensor powers of the hermitian line bundle:

% symmetry considerations, complex case
\begin{cor}\label{complexreduction}
Under the assumptions and in the notation of Theorem {\textup{\ref{main}}}, and modulo the indeterminacy $\underline{\underline{D}}_{l+1}^{\Gamma}+\ZZ^{\Gamma}[\![q]\!]+\RR\tilde G_{l+1}^{(N)}$, we have:
\begin{equation*}
f[S(\lambda)]\equiv\sum_{n\geq1}\left(\sum_{d|n}\left(\zeta^{-n/d}+(-1)^{l+1}\zeta^{n/d}\right)\xi\left(\eth\otimes\lambda^{d},\pi\right)\right)q^{n}.
\end{equation*}
Moreover, if $l$ is even, then $2f[S(\lambda)]=0$.
\end{cor}

\begin{proof} Since $[S(\lambda^{-1})]=-[S(\lambda)]$, it suffices to show
\begin{equation*}\label{symmetry}
\xi\left(\eth\otimes\lambda^{-d},\pi\right)\equiv(-1)^{l}\xi\left(\eth\otimes\lambda^{d},\pi\right)\mod\ZZ.
\end{equation*}
First, consider the case $l=2k+1$; it is well-known that untwisted spinors in dimension $4k+1$ carry a quaternionic structure if $k$ is odd, and a real structure if $k$ is even. Furthermore, the complex vector bundle $\lambda^{d}\oplus\lambda^{-d}$ admits both a real structure  as well as a quaternionic structure; therefore, we may put a quaternionic structure on the twisted spinor bundle $S_{4k+1}\otimes(\lambda^{d}\oplus\lambda^{-d})$, which implies that its Dirac spectrum is symmetric (since Clifford multiplication and therefore also the twisted Dirac operator anticommutes with the quaternionic structure) and that the kernel is even-dimensional (as a complex vector space). On the other hand,  the Chern--Simons term computed for $\lambda^{d}$ clearly differs from the one for $\lambda^{-d}$ by a factor of $(-1)$.

Now let $l=2k$; then, the (stably almost) complex vector bundle $TB$ and its conjugate bundle $\overline{TB}$ induce the same orientation on the $4k$--dimensional $(U,fr)$--manifold $B$, and we conclude
\begin{equation*}
\begin{split}
\int_{B^{4k}}Td\left(\nabla^{TB}\right)ch\left(\nabla^{\lambda^{d}}\right)&=\int_{B^{4k}}Td\left(\nabla^{\overline{TB}}\right)ch\left(\nabla^{\lambda^{-d}}\right)\\
&\equiv\int_{B^{4k}}Td\left(\nabla^{TB}\right)ch\left(\nabla^{\lambda^{-d}}\right)\mod\ZZ
\end{split}
\end{equation*}
by noting $Td(\overline{TB})=Td(TB)ch(\overline{L})$ and applying the integrality argument of (the proof of) Lemma \ref{integrality} to the relative class $\overline{L}\otimes\overline{\lambda}\in K(B,\partial B)$.
\end{proof}
 
%\begin{equation*}
%\begin{split}
%\left\langle Td(TB)ch(\lambda),[B,M]\right\rangle&=\left\langle Td(\overline{TB})ch(\overline{\lambda}),[B.M]\right\rangle\\
%&\equiv\left\langle Td(TB)ch(\overline{\lambda}),[B,M]\right\rangle\mod\ZZ
%\end{split}
%\end{equation*}

\begin{rmk*} Of course, the result $2f[S(\lambda)]=0$ for $l=2k$ also follows from  purely algebraic considerations, viz.\ that $\text{Ext}^{2,4k+2}_{MU_{*}MU}(MU_{*},MU_{*})$ consists entirely of 2-torsion.
\end{rmk*}

%%%%%%%%%%%%%%%%%%%%%%%%%%%%%%%%%%%%%%%%%%%%%%%%%%%

\section{The quaternionic situation}\label{quaternionic case}
There is an obvious analog in the quaternionic setting (identifying $Sp(1)\cong SU(2)$ as usual): The transfer map becomes
\begin{equation}\label{h-transfer}
S_{\HH}\colon\Omega^{fr}_{*}\left(\HH P^{\infty}\right)\rightarrow\Omega^{fr}_{*+3}
\end{equation}
and we have:
\begin{thm}\label{h-main}
Let $M$ be a closed manifold of dimension $2l-3>3$, equipped with a parallelism $\pi$, and let $\lambda_{\HH}$ be a quaternionic hermitian line bundle over~$M$. Then, choosing a unitary connection on $\lambda_{\HH}$, the  $f$--invariant \textup{(}at the level~$N$\textup{)} of the transfer $S(\lambda_{\HH})$  may be deduced from the following formula:
\begin{multline}
f[S(\lambda_{\HH})]\equiv\sum_{n\geq1}\left(\sum_{d|n}\xi\left(\eth\otimes d\left(\psi^{d}\lambda_{\HH}-2\right),\pi\right)\right)q^{n}\\
\mod\left\{\begin{array}{lcl}\underline{\underline{D}}_{l+1}^{\Gamma}+\ZZ^{\Gamma}[\![q]\!]&\mbox{if}&l\equiv0(2)\\
\underline{\underline{D}}_{l+1}^{\Gamma}+\ZZ^{\Gamma}[\![q]\!]+\RR\tilde G_{l+1}&\mbox{if}&l\equiv1(2) \end{array} \right.
\end{multline}
where $\tilde G_{l+1}=G_{l+1}+B_{l+1}/(2l+2)$ is the ordinary Eisenstein series with its constant term removed, and $\psi^{d}\lambda_{\HH}$ denotes the virtual hermitian vector bundle \textup{(}with the induced connection\textup{)} determined by the $d^{th}$ Adams operation.
\end{thm}
\begin{proof}
A straightforward computation reveals 
\begin{equation*}
\frac{{Ell}\left(\lambda_{\HH}\right)-1}{c_{2}\left(\lambda_{\HH}\right)}=g_{2}^{(N)}+\sum_{k\geq1}(-1)^{k+1}G_{2k+2}\frac{c_{2}^{k}\left(\lambda_{\HH}\right)}{(2k)!/2}
\end{equation*}
where $g_{2}^{(N)}$ is a modular form of level $N$ and weight two, and it satisfies $g_{2}^{(N)}\equiv\frac{1}{12}\mod\ZZ^{\Gamma}[\![q]\!]$ (cf.\ \cite{Bodecker:2011fu}). Clearly, we have
\begin{equation*}
\sum_{k\geq1}(-1)^{k+1}\tilde G_{2k+2}\frac{c_{2}^{k}\left(\lambda_{\HH}\right)}{(2k)!/2}=-\sum_{n\geq1}\left(\sum_{d|n}d\psi^{d}ch\left(\lambda_{\HH}-2\right)\right)q^{n}
\end{equation*}
which is the quaternionic analog of Proposition \ref{reduced genus}; noting that the assumption on the dimension of the base implies $g_{2}^{(N)}e_{\CC}[M]\equiv0$ (i.e.\ the $f$--invariant of the transfer on the trivial line $M\times\HH\rightarrow M$ vanishes), the desired result is obtained by making the obvious adjustments to the argument given in the complex situation.
\end{proof}

% two convenient descriptions using Chebyshev polynomials
\begin{rmk*}\label{spelled-out}
We can spell out the $K$--theory class $\psi^{d}\lambda_{\HH}$ as follows: Recall that the complex representation ring $RU(Sp(1))$ is polynomial in the fundamental two-dimensional representation, say $V$, with character $\chi_{V}\left(e^{it}\right)=2\cos(t)$; by the splitting principle, we have $\chi_{\psi^{d}V}\left(e^{it}\right)=2\cos(dt)$.

Now recall that Chebyshev polynomials of the first kind are recursively defined by
\begin{equation*}
T_{0}(x)=1,\quad T_{1}(x)=x,\quad T_{n+1}(x)=2xT_{n}(x)-T_{n-1}(x);
\end{equation*}
the multiple angle law for the cosine (for positive integers $d$) may then be expressed as $\cos(dt)=T_{d}\left(\cos(t)\right)$, showing that  $\psi^{d}\lambda_{\HH}$ expands in terms of (complex) tensor powers of $\lambda_{\HH}$ as $2T_{d}(\lambda_{\HH}/2)$ (which actually lies in $\ZZ[\lambda_{\HH}]$).

If, on the other hand, one prefers irreducible representations, one can make use of Chebyshev polynomials of the second kind: Their recurrence relation reads
\begin{equation*}
U_{0}(x)=1,\quad U_{1}(x)=2x,\quad U_{n+1}(x)=2xU_{n}(x)-U_{n-1}(x)
\end{equation*}
and one has $U_{n}(\cos(t))=\frac{\sin\left((n+1)t\right)}{\sin(t)}$, which may be identified with the character of the unique irreducible complex representation of dimension $n+1$. Assuming $d>1$, we have
\begin{equation*}
2\cos(dt)=2T_{d}(\cos(t))=U_{d}(\cos(t))-U_{d-2}(\cos(t)),
\end{equation*}
which yields the description of $\psi^{d}\lambda_{\HH}$ as the difference of the vector bundles associated to the irreducible representations of dimension $d+1$, $d-1$, respectively.
\end{rmk*}

% symmetry considerations, quaternionic case
\begin{cor}\label{quaternionicreduction}
Under the assumptions and in the notation of Theorem  \textup{\ref{h-main}}, and modulo the indeterminacy $\underline{\underline{D}}_{l+1}^{\Gamma}+\ZZ^{\Gamma}[\![q]\!]$, we have:
\begin{equation*}
f[S(\lambda_{\HH})]\equiv\left\{\begin{array}{lcl}\frac{1}{2}\sum_{n\geq1}\left(\sum_{2\nmid d|n}\dim\ker\left(\eth\otimes \psi^{d}\lambda_{\HH}\right)\right)q^{n}&\mbox{if}&l\equiv0(4),\\
0&\mbox{if}&l\equiv2(4). \end{array} \right.
\end{equation*}
\end{cor}
\begin{proof}
First of all, note that the $(4k+2)$--form part in the index density
\begin{equation*}
Td\left(\nabla^{TB}\right)ch\left(\psi^{d}\lambda_{\HH}-2\right)=\hat{A}\left(\nabla^{TB}\right)\exp\left(\textstyle{\frac{iF^{L}}{4\pi}}\right)ch\left(\psi^{d}\lambda_{\HH}-2\right)
\end{equation*}
is proportional to the curvature 2--form of the determinant line, $F^{L}$; since the latter vanishes when restricted to the framed manifold $M$, the same holds true for the Chern--Simons term keeping track of a change of connection on $TB|_{M}$. Next, recall that the (complex) spinor representations in dimension $4k+1$ come with a structure map $J_{S}$ which is real or quaternionic if $k$ is even or odd, respectively, and which anticommutes with Clifford multiplication. Similarly, the $Sp(1)$ represenations $\psi^{d}V$ carry a real (resp.\ quaternionic) structure map $J_{d}$ if $d$ is even (resp.\ odd). Therefore, for any $d$ and $k$, the spinor bundle $S_{4k+1}\otimes\psi^{d}\lambda_{\HH}$ carries a structure map (induced by $J_{S}\otimes J_{d}$) that anticommutes with the Dirac operator $\eth\otimes\psi^{d}\lambda_{\HH}$, so that its spectrum is symmetric. Therefore, we are left with the contribution from the kernels; reducing modulo two and taking into account that the structure map is quaternionic  for $l\equiv2\mod4$, the result follows.
\end{proof}

%%%%%%%%%%%%%%%%%%%%%%%%%%%%%%%%%%%%%%%%%%%%%%%%%%%

\section{Some examples}\label{some examples}

To start with, let $M$ be a closed parallelized manifold of dimension $2l-1$, and let $\lambda$ be the trivial complex line bunle $M\times\CC\to M$, endowed $\lambda$ with the natural hermitian metric and the tautological connection, $\nabla^{0}$. Then, for any $d$, we have $\xi(\eth\otimes\lambda^{d},\pi)\equiv e_{\CC}[M]$; applying Theorem \ref{main}, we immediately obtain a {\em rational} representative: 
\[f[S(\lambda)]\equiv -\tilde G_{1}^{(N)}e_{\CC}[M]\mod\underline{\underline{D}}_{l+1}^{\Gamma}+\ZZ^{\Gamma}[\![q]\!]+\RR\tilde G_{l+1}^{(N)}.\]
But here is an (admittedly artificial) example to illustrate that the full indeterminacy is needed: Let $\lambda$ be the trivial complex line bundle over the circle $M\cong\RR/\ZZ$, parametrized by the variable $t$. On We still use the natural hermitian metric on $\lambda$, but introduce the unitary connection $\nabla^{\lambda}=\nabla^{0}+2\pi i\epsilon {\rm d}t$, $0<\epsilon<1$. The standard parallelism of $\RR$ descends to the circle and induces the non-bounding spin structure. Then, twisted spinors are $\ZZ$--invariant complex-valued functions on the real line that are square-integrable on a fundamental domain, and the $\lambda$--twisted Dirac operator reads $\eth\otimes\lambda=-i\partial_{t}+2\pi\epsilon$. Therefore, its spectrum is $\{2\pi(k+\epsilon)\colon k\in\ZZ\}$, so
\[\eta(\eth\otimes\lambda,s)=(2\pi)^{-s}\left(\zeta_{H}(s,\epsilon)-\zeta_{H}(s,1-\epsilon)\right),\]
where $\zeta_{H}(s,x)$ is the Hurwitz zeta function. In particular, this implies that $(\eta(\eth\otimes\lambda)+\dim\ker(\eth\otimes\lambda))/2\equiv 1/2-\epsilon\mod\ZZ$; similarly, \[\frac{\eta(\eth\otimes\lambda^{d})+\dim\ker(\eth\otimes\lambda^{d})}{2}\equiv \frac{1}{2}-d\epsilon\mod\ZZ.\]
For dimensional reasons, there is no Chern--Simons term,  showing that the quantities $\xi(\eth\otimes\lambda^{d},\pi)$ may  very well be irrational (except for $\xi(\eth\otimes\lambda^{0},\pi)\equiv e_{\CC}[M]$); nevertheless,  applying Corollary \ref{complexreduction} to our example, we have
\[f[S(\lambda)]\equiv{\textstyle\frac{1}{2}}\tilde G_{1}^{(N)}+\epsilon\tilde G_{2}^{(N)}\mod\underline{\underline{D}}_{2}^{\Gamma}+\ZZ^{\Gamma}[\![q]\!]+\RR\tilde G_{2}^{(N)},\]
allowing us to recover $f[\eta^{2}]$ from our purely analytical computation.

% a `genuine' example: Hopf line on $S^{3}\times S^{2}$
For a more interesting example, let $G=S_{L}^3\times S_{R}^3$ be the Lie group  consisting of pairs $(q_{L},q_{R})$ of unit quaternions, $H\cong S^{1}$ the closed subgroup with Lie algebra spanned by $(0,k_{R})$, and $p\colon G\to G/H$ the natural projection. Decomposing $\mathfrak{g}\cong {\mathfrak{m}}\oplus\mathfrak{h}$ and specifying the ordered basis $\{(i_{L},0),(j_{L},0),(0,i_{R}),(0,j_{R}),(k_{L},0)\}$ for $\mathfrak{m}$, we treat $G/H$ as oriented Riemannian homogeneous space endowed with the normal homogeneous metric induced by the obvious bi-invariant metric on $G$.
Clearly, the isotropy representation $\iota\colon H\rightarrow {\rm Aut}({\mathfrak m})$ extends to a representation $\rho\colon G\rightarrow {\rm Aut}({\mathfrak m})$, where $\rho\cong1\oplus1\oplus Ad_{S^{3}_{R}}$ and $Ad_{S^{3}_{R}}$ is considered as a $G$--module. Thus, $T(G/H)\cong G\times_{\iota}{\mathfrak{m}}$ may be trivialized using sections $$s_{i}\colon G/H\rightarrow T(G/H)\cong G\times_{\iota}{\mathfrak{m}},\ gH\mapsto [g,\rho(g)^{-1}e_{i}];$$ by construction, this parallelism $\pi$ is compatible with the chosen metric. Now let $\lambda\to G/H$ be the tautological homogeneous complex line bundle associated to the standard representation of $H\cong U(1)$, and endow it with the natural hermitian metric and connection.  Note that, with our choice of orientation, $c_{1}(\lambda)$ restricts to the {\em positive} generator of $H^{2}(S^{3}_{R}/S^{1};\ZZ)$.

\begin{prop*} Let $\lambda\to G/H$ be this tautological  hermitian line bundle  and let $G/H$ be parallelized as above. Then the transfer  $S(\lambda)$ represents the non-trivial element $\nu^{2}\in\pi^{S}_{6}S^{0}\cong\ZZ/2$.
\end{prop*}

\begin{proof}
 To make the parallelism more explicit, observe that if $Ad_{S^{3}_{R}}(q^{-1}_{R})$ is expressed as an orthogonal matrix w.r.t.\ the basis $\{i_{R},j_{R},k_{R}\}$, its columns may be identified with right-invariant vector fields corresponding to $i_{R},j_{R},k_{R}$; furthermore, deleting the last row amounts to projecting onto $S^{3}_{R}/S^{1}$.
Now for $v\in\mathfrak{g}$, let $l(v)$, $r(v)$ denote the left invariant and right invariant extensions, respectively, and consider the following vector fields on $G/H$:
\begin{gather*}
L_{1}=p_{*}l(i_{L},0),\ L_{2}=p_{*}l(j_{L},0),\ L_{3}=p_{*}l(k_{L},0),\\ K_{1}=p_{*}r(0,i_{R}),\ K_{2}=p_{*}r(0,j_{R}),\ K_{3}=p_{*}r(0,k_{R}).
\end{gather*}
Then, defining three real functions $y_{i}$ by $q_{R}k_{R}{q_{R}}^{-1}=y_{1}i_{R}+y_{2}j_{R}+y_{3}k_{R}$, we may express the trivializing sections $s_{1},\dots,s_{5}$ as
$$L_{1}, L_{2}, w_{1}=K_{1}+y_{1}L_{3}, w_{2}=K_{2}+y_{2}L_{3}, w_{3}=K_{3}+y_{3}L_{3},$$
respectively; also note that $L_{3}=y_{1}w_{1}+y_{2}w_{2}+y_{3}w_{3}$.
In particular, this yields the following commutation relations:
\begin{gather*}
[L_{1},L_{2}]=2L_{3},\quad [L_{1},w_{l}]=-2y_{l}L_{2},\quad [L_{2},w_{l}]=2y_{l}L_{1},\\
[w_{1},w_{2}]=-2(w_{3}+y_{3}L_{3}),\quad [w_{2},w_{3}]=-2(w_{1}+y_{1}L_{3}),\\ [w_{3},w_{1}]=-2(w_{2}+y_{2}L_{3}).
\end{gather*}
Thus, w.r.t.\ the global coframe $\theta=(L_{1}^{*},L_{2}^{*},w_{1}^{*},w_{2}^{*},w_{3}^{*})^{tr}$, Cartan's structure equation ${\rm d}\theta=-\omega\wedge\theta$ reads:
\begin{equation*}
{\rm d}\theta=-\begin{pmatrix}0&-L_{3}^{*}&y_{1}L_{2}^{*}&y_{2}L_{2}^{*}&y_{3}L_{2}^{*}\\L_{3}^{*}&0&-y_{1}L_{1}^{*}&-y_{2}L_{1}^{*}&-y_{3}L_{1}^{*}\\-y_{1}L_{2}^{*}&y_{1}L_{1}^{*}&0&2w_{3}^{*}-2y_{3}L_{3}^{*}&2y_{2}L_{3}^{*}-2w_{2}^{*}\\ -y_{2}L_{2}^{*}&y_{2}L_{1}^{*}&2y_{3}L_{3}^{*}-2w_{3}^{*}&0&2w_{1}^{*}-2y_{1}L_{3}^{*}\\ -y_{3}L_{2}^{*}&y_{3}L_{1}^{*}&2w_{2}^{*}-2y_{2}L_{3}^{*}&2y_{1}L_{3}^{*}-2w_{1}^{*}&0 \end{pmatrix}\wedge\theta
\end{equation*}
Using this, we proceed to determine the contribution from the Chern--Simons term: A straightforward computation reveals $\tr(\omega{\rm d}\omega)=12L_{1}^{*}\wedge L_{2}^{*}\wedge L_{3}^{*}$ and $\tr(\omega^{3})=-6L_{1}^{*}\wedge L_{2}^{*}\wedge L_{3}^{*}$, hence
$$\int_{M}cs(\hat A,\pi)c_{1}(\lambda^{d})=\left(-\frac{1}{24}\right)\left(-\frac{1}{8\pi^{2}}\right)\int_{M}\tr(\omega{\rm d}\omega+{\textstyle\frac{2}{3}}\omega^{3})c_{1}(\lambda^{d})=\frac{d}{12}.$$
Next, we have to deal with the $\eta$--term: Clearly, the spin structure on $G/H$ is unique, hence the Dirac spectrum depends only on the Riemannian metric; since the latter is normal homogeneous, we could decompose the space of spinors into finite-dimensional $G$--modules and determine the spectrum explicitly. However, we may avoid this calculation by using the following argument: As explained in \cite[Section 2]{Goette:1999cq}, twisted Dirac operators on homogeneous fibrations exhibit a certain amount of  spectral symmetry; in particular,  this is true for the ($G$--equivariant) Riemannian submersion $G/H\to S^{3}_{L}$, and for $d>0$, the computation of the $\eta$--invariant of $\eth\otimes\lambda^{d}$  boils down to a computation of the $\eta$--invariant of ($d$ copies of) the untwisted Dirac operator on $S^{3}_{L}$ (with its round metric). But the latter operator has trivial kernel and its spectrum is symmetric (see e.g.\ \cite{Bar:1992kl}), whence the spectral contribution to $\xi(\eth\otimes\lambda^{d},\pi)$ is trivial.

Applying Corollary \ref{complexreduction}, we therefore conclude
$$f[S(\lambda)]\equiv\frac{1}{12}\tilde G_{2}^{(N)}\equiv\frac{1}{2}\left(\tilde G_{2}^{(N)}\right)^{2}\mod\underline{\underline{D}}_{4}^{\Gamma_{1}(N)};$$
at odd levels, this indeed represents $f[\nu^{2}]$ (see e.g.\ \cite{Bodecker:2008pi}).
\end{proof}

% a quaternionic transfer representing im(J)
Similarly, we obtain a non-trivial example for the quaternionic transfer by considering the tautological quaternionic line bundle $\lambda_{\HH}\to S^{1}\times \HH P^{1}$, treating $S^{1}\times \HH P^{1}\approx S^{1}\times Sp(2)/Sp(1)^{\times2}$ as a normal homogeneous space. Clearly, the real 5-dimensional irreducible $Sp(2)$--module (which may be realized by the conjugation action on traceless self-adjoint two-by-two quaternionic matrices) provides an extension of the isotropy representation, and we use this to trivialize the tangent bundle by orthonormal sections.

\begin{prop*} Let $\lambda_{\HH}$ be this tautological  quaternionic line bundle, equipped with its natural hermitian metric and connection,  and let $S^{1}\times Sp(2)/Sp(1)^{\times2}$ be parallelized as above. Then the transfer  $S(\lambda_{\HH})$ represents the generator of ${\rm im}(J\colon\pi_{8}SO\rightarrow\pi_{8}^{S}S^{0})$, i.e.\ $[S(\lambda_{\HH})]=\eta\sigma\in\pi_{8}^{S}S^{0}\cong\ZZ/2\oplus\ZZ/2$.
\end{prop*}
\begin{proof} There are two inequivalent spin structures on $S^{1}\times \HH P^{1}$, but a moment's thought reveals that the spin structure induced by our choice of parallelism restricts to the non-bounding one on the circle. Moreover, the projection onto $S^{1}$ is a Riemannian submersion, and we have
\[\dim\ker\left(\eth\otimes \psi^{d}\lambda_{\HH}\right)\equiv{\rm ind}\left(\eth^{+}\otimes \psi^{d}\lambda_{\HH}|_{\HH P^{1}}\right)=d^{2}\equiv d^{3}\mod2,\]
where $\eth^{+}\otimes \psi^{d}\lambda_{\HH}|_{\HH P^{1}}$ is the (chiral) Dirac operator on the fiber, and  its index is readily evaluated using the cohomological version of the index theorem. Applying Corollary \ref{quaternionicreduction}, we therefore conclude
\[f[S(\lambda_{\HH})]\equiv\frac{1}{2}\sum_{n\geq1}(\sum_{d|n}d^{3})q^{n}\equiv\frac{1}{2}\left(\frac{E_{4}-1}{240}\right)\mod\underline{\underline{D}}_{5}^{\Gamma}+\ZZ^{\Gamma}[\![q]\!];\]
at any odd level, this is a representative of $f[\eta\sigma]$. But since $\pi_{8}^{S}S^{0}\cong\ZZ/2\oplus\ZZ/2$ is 2--torsion and known to contain no non-trivial elements of filtration $\geq3$, the claim follows.
\end{proof}

%%%%%%%%%%%%%%%%%%%%%%%%%%%%%%%%%%%%%%%%%%%%%%%%%%
\appendix

\section{Stabilization}\label{stabilization}
% a modification using a background metric
In order to deal with a merely stably parallelizable manifold $M$, we may proceed as follows: First of all, fix a Riemannian metric $\langle\cdot,\cdot\rangle$ on $M$, and let $\nabla^{LC}$ be its Levi-Civita connection. Furthermore, let $\varepsilon$ denote a trivialized real line bundle over $M$, equipped with the tautological metric and compatible connection (by declaring the trivializing section to be of unit length and parallel). Consequently, we get induced metrics and connections on Whitney sums of these bundles; working with a fixed representative $\varepsilon^{\oplus s'}\oplus TM$ of the stable tangent bundle, we denote these by $\langle\cdot,\cdot\rangle_{sTM}$ and  $\nabla^{sLC}$, respectively. Without loss of generality, a trivialization of $\varepsilon^{\oplus s'}\oplus TM$ may be assumed to be given by an ordered set of $s'+\dim M$ orthonormal sections (w.r.t.\ $\langle\cdot,\cdot\rangle_{sTM}$), and there is a tautological connection $\nabla^{s\pi}$ annihilating this frame. Since the given (stable) frame induces an orientation and a spin structure on the Riemannian manifold $M$, we have a well-defined Dirac operator on $M$; moreover, twisting by a hermitian vector bundle $\lambda$ with unitary connection $\nabla^{\lambda}$, we may form
\[\xi(\eth\otimes\lambda,s\pi)=\frac{\eta(\eth\otimes\lambda)+\dim\ker(\eth\otimes\lambda)}{2}-\int_{M} cs(\hat{A};\nabla^{sLC},\nabla^{s\pi})ch(\nabla^{\lambda}).\]
Clearly, this is the appropriate replacement for $\xi(\eth\otimes\lambda,\pi)$, and our results hold verbatim. As an example, let us give an analytic proof of the following classical result:

% Steer--Wood
\begin{prop*}[\cite{Steer:1976vn,Wood:1976yi}] Twisting the normal framing induced by the left invariant trivialization of the tangent bundle of $SU(3)$ by the realification of the defining representation, $\rho\colon SU(3)\rightarrow SO(6)$, yields a representative of the generator of ${\rm im}(J\colon\pi_{8}SO\rightarrow\pi_{8}^{S}S^{0})$, i.e.\ $[SU(3),\rho]=\eta\sigma\in\pi_{8}^{S}S^{0}$.
\end{prop*}
\begin{proof}
The natural inclusion of $H=SU(2)$ into $G=SU(3)$ defines a fiber bundle $p\colon G\rightarrow G/H=M\approx S^{5}$, which may be identified with the sphere bundle of the tautological quaternion line $\lambda_{\HH}\rightarrow M$. We equip $G$ with a bi-invariant metric (which is unique up to scale), giving rise to a naturally reductive decomposition $\mathfrak{g}=\mathfrak{m}\oplus\mathfrak{h}$ and endowing $M$ with a normal homogeneous metric. Making use of $\rho|H\cong1\oplus\iota$, where $\iota\colon H\rightarrow {\rm Aut}(\mathfrak{m})$ is the isotropy representation, we obtain a stable tangential framing $M\times\RR^{6}\rightarrow \varepsilon\oplus T(G/H)$; applying the transfer construction, the induced framing of $S(\lambda_{\HH})\approx G$ obviously describes the element $[SU(3),\rho]$. Thus, we may compute $f[SU(3),\rho]$ using (the modified version of) Corollary \ref{quaternionicreduction}. To this end, let $\kappa\colon H\rightarrow {\rm Aut} V_{2k+2}$ denote an irreducible complex $H$--representation of dimension $2k+2$ (which is unique up to isomorphism), and let $\Sigma^{(\kappa)}M$ be the spinor bundle twisted by the vector bundle associated to $\kappa$. Now, the $\kappa$--twisted Dirac operator $\eth^{\kappa}$ may be continously deformed into what is called the reductive Dirac operator $\tilde{D^{\kappa}}$ in \cite{Goette:1999cq}; the latter is still $G$--equivariant, hence it respects the decomposition of the space of spinors
$$L^{2}(\Sigma^{\kappa}M)\cong\overline{\bigoplus_{\gamma\in\hat G}W^{\gamma}\otimes {\rm Hom}_{H}(W^{\gamma},\Sigma_{{\mathfrak{m}}}\otimes V_{2k+2})},$$
giving rise to elements $\tilde{\leftidx{^{\gamma}}{D}{^{\kappa}}}\in{\rm End}\left({\rm Hom}_{H}(W^{\gamma},\Sigma_{{\mathfrak{m}}}\otimes V_{2k+2})\right)$ by restriction:
\[id_{W^{\gamma}}\otimes \tilde{\leftidx{^{\gamma}}{D}{^{\kappa}}}=\tilde{D^{\kappa}}|_{W^{\gamma}\otimes {\rm Hom}_{H}(W^{\gamma},\Sigma_{{\mathfrak{m}}}\otimes V_{2k+2})}.\]
Moreover, by \cite[Lemma 1.17]{Goette:1999cq}, we have \[(\tilde{\leftidx{^{\gamma}}{D}{^{\kappa}}})^{2}=||\gamma+\rho_{G}||^{2}-||\kappa+\rho_{H}||^{2}.\]
We can now determine $\dim\ker(\eth^{\kappa})\equiv\dim\ker(\tilde{D^{\kappa}})\equiv\dim\ker(\tilde{D^{\kappa}})^{2}\mod2$ as follows: Up to an irrelevant scaling factor, the root lattice of $SU(3)$ may be identified with the Eisenstein lattice $\mathcal{E}={\rm span}_{\ZZ}\{1,{\textstyle\frac{1}{2}+\frac{\sqrt{-3}}{2}}\}\subset\CC$, and the fundamental dual Weyl chamber may be chosen such that the weights lying in its closure  are given by $m(\frac{1}{2}+\frac{i\sqrt3}{6})+n(\frac{1}{2}-\frac{i\sqrt{3}}{6})$, $m,n\in\NN_{0}$; under this identification, the Weyl vector of $G$ is given by $\rho_{G}=1\in\mathcal{E}$, and we have $\rho_{H}+\kappa=k+1$. By symmetry, it suffices to deal with the contributions from the representations labelled by weights on the real line, i.e.\ with the case $\gamma=k\in\mathcal{E}$; for these, the Weyl character formula implies $\dim W^{\gamma=k}=(k+1)^{3}$. Taking into account that $\Sigma_{\mathfrak{m}}\cong 2\oplus V_{2}$ (since the natural inclusion $SU(2)\cong Sp(1)\subset Sp(2)\cong Spin(5)$ provides a lift of the isotropy representation $\iota$) and $V_{2}\otimes V_{2k+2}\cong V_{2k+1}\oplus V_{2k+3}$, we conclude that $\dim\ker(\eth^{\kappa})\equiv k+1\mod 2$. Thus, $\ker(\eth\otimes\psi^{d}\lambda_{\HH})$ is odd-dimensional for any odd $d>0$, and therefore
\[f[SU(3),\rho]\equiv\frac{1}{2}\sum_{n\geq1}(\sum_{d|n}d^{3})q^{n}\equiv\frac{1}{2}\left(\frac{E_{4}-1}{240}\right)\mod\underline{\underline{D}}_{5}^{\Gamma}+\ZZ^{\Gamma}[\![q]\!],\]
which indeed represents $f[\eta\sigma]$ (compare Section \ref{some examples}).\end{proof}

\bibliography{refbib_edited}
\end{document}